\documentclass[12pt]{amsart}

\usepackage{amsmath,amssymb,amsthm,mathtools}
\usepackage{enumitem}
\usepackage{mathrsfs}
\usepackage{xcolor}
\usepackage{aliascnt}
\definecolor{refblue}{RGB}{0,27,126}
\definecolor{citegreen}{RGB}{0,126,68}
\usepackage[
  colorlinks=true,
  linkcolor=refblue,
  citecolor=citegreen,
  urlcolor=purple,
  filecolor=purple]{hyperref}
\hypersetup{pdftitle={Isoperimetric bounds for the Steklov and weighted Neumann eigenvalues in dimensions three through six},pdfauthor={Daguang Chen, Hao Liu, Chengxi Yang}}
\usepackage[nameinlink,capitalize,noabbrev]{cleveref}
\usepackage{microtype}
\usepackage{needspace}
\usepackage[a4paper, left=2.7cm, right=2.7cm, top=2.5cm, bottom=2.5cm]{geometry}

\numberwithin{equation}{section}

\newtheorem{theorem}{Theorem}[section]

\newaliascnt{proposition}{theorem}
\newtheorem{proposition}[proposition]{Proposition}
\aliascntresetthe{proposition}

\newaliascnt{lemma}{theorem}
\newtheorem{lemma}[lemma]{Lemma}
\aliascntresetthe{lemma}

\newaliascnt{corollary}{theorem}
\newtheorem{corollary}[corollary]{Corollary}
\aliascntresetthe{corollary}

\theoremstyle{definition}
\newaliascnt{definition}{theorem}

\aliascntresetthe{definition}
\newtheorem*{definition*}{Definition}

\newaliascnt{remark}{theorem}
\newtheorem{remark}[remark]{Remark}
\aliascntresetthe{remark}

\crefname{theorem}{Theorem}{theorems}
\Crefname{theorem}{Theorem}{Theorems}
\crefname{proposition}{Proposition}{propositions}
\Crefname{proposition}{Proposition}{Propositions}
\crefname{lemma}{Lemma}{lemmas}
\Crefname{lemma}{Lemma}{Lemmas}
\crefname{corollary}{Corollary}{corollaries}
\Crefname{corollary}{Corollary}{Corollaries}
\crefname{definition}{Definition}{definitions}
\Crefname{definition}{Definition}{Definitions}
\crefname{remark}{Remark}{remarks}
\Crefname{remark}{Remark}{Remarks}

\newcommand{\R}{\mathbb R}
\newcommand{\Sph}{\mathbb S}
\newcommand{\Om}{\Omega}
\newcommand{\dd}{\,d}
\newcommand{\Span}{\operatorname{span}}
\newcommand{\supp}{\operatorname{supp}}
\newcommand{\Mcont}{\mathcal M_c^+}
\newcommand{\qform}{\mathfrak q}

\title[Isoperimetric bounds for Steklov and weighted Neumann eigenvalues]
{Isoperimetric bounds for the weighted Neumann and Steklov  eigenvalues in dimensions three through six}

\author{Daguang Chen}
\email{dgchen@tsinghua.edu.cn}
\author{Hao Liu}
\email{hao-liu25@mails.tsinghua.edu.cn}
\author{Chengxi Yang}
\email{ycx24@mails.tsinghua.edu.cn}
\address{Department of Mathematical Sciences, Tsinghua University, Beijing 100084, P. R. China.}

\subjclass[2020]{Primary 35P15; Secondary 35J25, 49R05, 58E20}
\keywords{Weighted Neumann eigenvalues, Steklov eigenvalues, Harmonic maps, Isoperimetric inequalities}
\thanks{The authors were supported by NSFC grant No. 11831005 and NSFC-FWO W2521103.}
\date{\today}

\begin{document}
\begin{abstract}
    For $3\le n\le6$, we determine the sharp isoperimetric bound for the mass-normalized first weighted Neumann eigenvalue on bounded Euclidean domains. For every finite nonnegative nonatomic measure $\mu$ on $\overline\Om$,
    \begin{equation*}
        \overline\lambda_1^N(\Om,\mu)\le\frac{n(n-1)}{n-2}\omega_n\sin^2\vartheta_n\left(\frac{|\Om|}{\omega_n}\right)^{(n-2)/n},
    \end{equation*}
    where $\omega_n$ is the volume of the unit ball and $\vartheta_n\in(\pi/2,\pi)$ is the angle at the first stationary radius of the regular rotational harmonic-map profile. The bound is attained on balls by a smooth positive radial density.  In fact, we affirmatively resolve Question~1.12 posed by Vinokurov \cite{Vinokurov2026}. Furthermore, for admissible domains with finite boundary measure, we also prove the sharp strict Steklov bound
    \begin{equation*}
        \sigma_1(\Om)|\partial\Om|\,|\Om|^{(2-n)/n}<\frac{n(n-1)}{n-2}\omega_n^{2/n}\sin^2\vartheta_n,
    \end{equation*}
    whose constant is approached by $C^1$ perforated domains.
\end{abstract}
\maketitle

\section{Introduction}\label{sec:introduction}

The first positive Neumann eigenvalue has a long history in isoperimetric spectral geometry. Szeg\H{o} \cite{Szego1954} proved that the disk maximizes it among simply connected planar domains of fixed area, and Weinberger \cite{Weinberger1956} extended the ball comparison to higher dimensions. For a detailed survey of progress and related extremal problems for Neumann eigenvalues, see \cite{Henrot2006}. The related Steklov problem is
\begin{equation*}
    \begin{cases}
        \ \Delta u=0 &\text{in }\Om,\\
        \ \partial_\nu u=\sigma u &\text{on }\partial\Om.
    \end{cases}
\end{equation*}
Here $\partial_\nu$ denotes the outward normal derivative. For a connected Lipschitz domain, the Steklov spectrum is discrete and can be ordered as
\begin{equation*}
    0=\sigma_0(\Om)<\sigma_1(\Om)\le\sigma_2(\Om)\le\cdots\nearrow\infty,
\end{equation*}
with eigenvalues repeated according to multiplicity. Weinstock \cite{Weinstock1954} proved that the disk maximizes the first positive Steklov eigenvalue among simply connected planar domains of fixed perimeter, while Brock \cite{Brock2001} established the ball comparison under a volume constraint in Euclidean space. In dimensions $n\ge3$, Bucur, Ferone, Nitsch, and Trombetti \cite{BucurFeroneNitschTrombetti2021} proved a boundary-area comparison for convex domains; Fraser and Schoen \cite{FraserSchoen2019} showed that the ball need not maximize the same normalized eigenvalue among contractible domains. The planar higher-eigenvalue inequalities of Hersch, Payne, and Schiffer \cite{HerschPayneSchiffer1974} were shown to be sharp by Girouard and Polterovich \cite{GirouardPolterovich2010}. 

The normalization relevant here is
\begin{equation*}
    \sigma_1(\Om)|\partial\Om|\,|\Om|^{(2-n)/n}.
\end{equation*}
It is invariant under spatial dilation. Colbois, El Soufi, and Girouard \cite{ColboisElSoufiGirouard2011} obtained uniform bounds for this quantity and its higher-eigenvalue analogues. Karpukhin and M\'etras \cite{KarpukhinMetras2022} identified the role of a joint boundary-volume normalization in the theory of extremal Steklov metrics. For more on eigenvalue estimates, see \cite{GirouardPolterovich2017,ColboisGirouardGordonSher2024,LiWangWu2025,GuLiWan2025,GuLiWan2026} and references therein.

Variational eigenvalues associated with measures provide a common setting for interior densities and boundary measures. Grigor'yan, Netrusov, and Yau \cite{GrigoryanNetrusovYau2004} established general bounds for measure-dependent eigenvalues, and Kokarev \cite{Kokarev2014} studied their optimization on surfaces. Freitas and Laugesen \cite{FreitasLaugesen2020} connected planar Neumann and Steklov inequalities through a Robin family. Girouard, Henrot, and Lagac\'e \cite{GirouardHenrotLagace2021} related Neumann spectra to homogenized limits of Steklov spectra, while Girouard, Karpukhin, and Lagac\'e \cite{GirouardKarpukhinLagace2021} developed continuity and approximation results for variational eigenvalues of measures.

Vinokurov \cite{Vinokurov2026} proved that, for admissible $\Om$ with $n\ge3$ and $|\Om|=\omega_n$, where $\omega_n=|B_1(0)|$, the mass-normalized first weighted Neumann eigenvalue satisfies
\begin{equation*}
    \overline\lambda_1^N(\Om,\mu)\le\frac{n(n-1)}{n-2}\omega_n.
\end{equation*}
For $n\ge7$, this bound is sharp, with the ball and a singular density proportional to $|x|^{-2}$ as a maximizing pair. His corresponding strict Steklov estimate holds for all $n\ge3$ and has a sharp constant when $n\ge7$. For $3\leq n\leq 6$, he posed the following question:
\begin{quote}
    \textbf{Question 1.12.} \cite{Vinokurov2026} \emph{Let $\Omega\subset\mathbb R^d$ range over admissible domains with $3\le d\le6$. Does the following equality hold?}
    \begin{equation*}
        \sup_{|\Omega|=|\mathbb B^d|}\Lambda_1^N(\Omega)=\Lambda_1^N(\mathbb B^d).
    \end{equation*}
\end{quote}
In this paper, we answer this question affirmatively, determine the sharp value and the maximizing pairs, and establish the corresponding sharp Steklov inequality in dimensions three through six.

\subsection{Variational framework and normalization}
Let $\Om\subset\R^n$ be a nonempty bounded open set and write $\mathscr D(\Om)=H^1(\Om)\cap C^0(\overline\Om)$. The admissible domain class is specified in \cite[Definition~1.1]{Vinokurov2026}.

\begin{definition*}[Admissible domain]
    We call $\Om$ \emph{admissible} if the following conditions hold:
    \begin{enumerate}[label=\textup{(\roman*)},leftmargin=*]
        \item $\mathscr D(\Om)$ is dense in $H^1(\Om)$;
        \item the embedding $H^1(\Om)\hookrightarrow L^2(\Om)$ is compact;
        \item if $\{\Om_i\}_i$ are the connected components of $\Om$, then
        \begin{equation*}
            \overline\Om=\bigsqcup_i\overline{\Om_i}.
        \end{equation*}
    \end{enumerate}
\end{definition*}

Let $\Mcont(\overline\Om)$ denote the finite nonnegative nonatomic Radon measures on $\overline\Om$. For a nonzero such measure, set
\begin{equation}\label{eq:rayleigh}
    \lambda_1^N(\Om,\mu)
    =\inf_{\substack{v\in\mathscr D(\Om),\ \int_{\overline\Om}v\,d\mu=0\\
    \int_{\overline\Om}v^2\,d\mu>0}}
    \frac{\int_\Om|\nabla v|^2\dd x}{\int_{\overline\Om}v^2\,d\mu}.
\end{equation}
The variational levels are indexed from zero, counting multiplicity. For a smooth positive density $\rho$, this is the first level after the constants for $-\Delta v=\lambda\rho v$ with Neumann boundary condition. It may vanish when $\Om$ is disconnected.

The mass-normalized eigenvalue and the associated domain supremum are defined by
\begin{equation*}
    \overline\lambda_1^N(\Om,\mu)
    =\mu(\overline\Om)\lambda_1^N(\Om,\mu),\qquad
    \Lambda_1^N(\Om)=\sup_{0\ne\mu\in\Mcont(\overline\Om)}
    \overline\lambda_1^N(\Om,\mu).
\end{equation*}
We set $\overline\lambda_1^N(\Om,0)=0$. For nonzero $\mu$, the mass-normalized eigenvalue is unchanged when $\mu$ is multiplied by a positive constant.

For an admissible domain with $|\partial\Om|=\mathcal H^{n-1}(\partial\Om)<\infty$, the boundary-measure formulation of the Steklov problem gives
\begin{equation}\label{eq:steklov-measure}
    \overline\lambda_1^N\bigl(\Om, \mathcal H^{n-1}\!\restriction_{\partial\Om}\bigr)=\sigma_1(\Om)|\partial\Om|.
\end{equation}
For Lipschitz domains, this agrees with the classical Steklov eigenvalue.

\subsection{Main results}
Write $\mathbb{B}^n=B_1(0)$ and $\omega_n=|\mathbb{B}^n|$. For $3\le n\le6$, the regular north-pole trajectory of J\"ager--Kaul \cite[Section~2]{JaegerKaul1983} has a first turning point. Normalize its radius to one and denote the resulting profile by $f$. Thus
\begin{equation*}
        \begin{cases}
            \ \displaystyle f''+\frac{n-1}{r}f'-\frac{n-1}{r^2}\sin f\cos f=0, \quad 0<r<1,\\
            \ f(0)=0,\qquad f'>0\ \text{on }(0,1),\qquad f'(1)=0.
        \end{cases}
    \end{equation*}
The profile is smooth at the origin. Put $\vartheta_n=f(1)\in(\pi/2,\pi)$ and define
\begin{equation*}
    \begin{aligned}
        U(r,\theta)&=\big(\sin f(r)\theta,\cos f(r)\big),\\
        Q(r)&=|\nabla U|^2=f'(r)^2+\frac{n-1}{r^2}\sin^2 f(r),
    \end{aligned}
\end{equation*}
with $Q(0)=nf'(0)^2$. We also write $Q(x)=Q(|x|)$ on $\mathbb{B}^n$.

\Needspace{12\baselineskip}
\begin{theorem}\label{thm:main} Let $\Omega \subset \R^n$ be an admissible domain such that $|\Omega| = |\mathbb{B}^n| = \omega_{n}$ and $3\le n\le6$.
Then every $\mu\in\Mcont(\overline\Om)$ satisfies
    \begin{equation}\label{eq:normalized-main}
        \overline\lambda_1^N(\Om,\mu)\le\overline\lambda_1^N(\mathbb{B}^n,Q\dd x)=\frac{n(n-1)}{n-2}\omega_n\sin^2\vartheta_n.
    \end{equation}
    If equality holds, then for some $c\in\R^n$ and $a>0$,
    \begin{equation*}
        |\Om\mathbin\triangle B_1(c)|=0,\qquad \overline\Om=\overline{B_1(c)},\qquad \mu=aQ(|x-c|)\dd x.
    \end{equation*}
    If $\Om$ has Lipschitz boundary, equality holds if and only if $\Om=B_1(c)$ and $\mu=aQ(|x-c|)\dd x$ for some $c\in\R^n$ and $a>0$.
\end{theorem}
\begin{remark}\label{rem:vinokurov-comparison}
    Vinokurov's upper bound \cite[Theorems~1.5 and~1.9]{Vinokurov2026} has constant $\frac{n(n-1)}{n-2}\omega_n$. It is sharp for $n\ge7$, with a singular maximizing density proportional to $|x|^{-2}$. For $3\le n\le6$, that density does not attain the equator-map comparison value \cite[(1.6)]{Vinokurov2026}; here the smooth density $Q$ attains the smaller constant in \eqref{eq:normalized-main}. This difference reflects the dimension threshold for the regular rotational profile: it has a first stationary radius when $3\le n\le6$, but none when $n\ge7$ \cite[Section~2]{JaegerKaul1983}.
\end{remark}
\begin{remark}\label{rem:weighted-neumann-consequences}
   Among admissible domains of volume $\omega_n$, we obtain
    \begin{equation*}
        \sup_{\substack{\Om\text{ admissible}\\|\Om|=\omega_n}}\Lambda_1^N(\Om)=\Lambda_1^N(\mathbb{B}^n)=\frac{n(n-1)}{n-2}\omega_n\sin^2\vartheta_n.
    \end{equation*}
    This answers Vinokurov's Question~1.12 affirmatively. For the scaling property of Neumann eigenvalue, \cref{thm:main} gives
    \begin{equation}\label{eq:main-bound}
        \overline\lambda_1^N(\Om,\mu)\le\frac{n(n-1)}{n-2}\omega_n\sin^2\vartheta_n\left(\frac{|\Om|}{\omega_n}\right)^{(n-2)/n}
        \qquad\bigl(\mu\in\Mcont(\overline\Om)\bigr).
    \end{equation}
    For Lipschitz domains, equality holds precisely for $\Om=B_R(c)$ and $\mu=aQ(|x-c|/R)\dd x$, with $a>0$. Without boundary regularity, equality implies $|\Om\mathbin\triangle B_R(c)|=0$ and $\overline\Om=\overline{B_R(c)}$, together with the same density.
\end{remark}

Applying \eqref{eq:main-bound} to the boundary measure gives a strict Steklov inequality by the equality characterization in \cref{thm:main}. The Steklov approximation theorem of Girouard, Karpukhin, and Lagac\'e \cite{GirouardKarpukhinLagace2021}, in the form of \cite[Proposition~1.4]{Vinokurov2026}, shows that its constant is sharp.
\begin{corollary}\label{cor:steklov}
    For $3\le n\le6$, every admissible domain $\Om\subset\R^n$ with $|\partial\Om|<\infty$ satisfies
    \begin{equation}\label{eq:steklov-bound}
        \sigma_1(\Om)|\partial\Om|\,|\Om|^{(2-n)/n}<\frac{n(n-1)}{n-2}\omega_n^{2/n}\sin^2\vartheta_n.
    \end{equation}
    The constant is sharp: there are bounded $C^1$ domains $\Om_j\subset \mathbb{B}^n$ with $|\Om_j|\to\omega_n$ whose normalized Steklov values converge to the right-hand side.
\end{corollary}
\begin{remark}\label{rem:steklov-comparison}
    Vinokurov's first Steklov bound \cite[Corollary~1.11]{Vinokurov2026} has constant $\frac{n(n-1)}{n-2}\omega_n^{2/n}$. Since $0<\sin^2\vartheta_n<1$, \cref{cor:steklov} gives the sharp constant in dimensions three through six, smaller by the factor $\sin^2\vartheta_n$.
\end{remark}

The paper is organized as follows. 
\cref{sec:model} records the rotational comparison map and its energy properties. 
\cref{sec:comparison} gives the auxiliary results for the model and domain comparison. 
\cref{sec:main-proof} proves \cref{thm:main}, including its equality characterization, and establishes the sharp Steklov inequality in \cref{cor:steklov}.

\section{The rotational comparison map}\label{sec:model}
We first recall the regular radial profile that produces the weighted Neumann model on the ball. We then extend the map to $\R^n$ and establish the monotonicity and dilation estimate used in the domain comparison.

\subsection{The regular profile and the dimension threshold}
For $x=r\theta$ with $\theta\in\Sph^{n-1}$, the rotational ansatz $U(r,\theta)=(\sin f(r)\theta,\cos f(r))$ reduces the harmonic-map equation to the radial equation below. The regular branch issuing from the north pole is unique up to radial dilation, and its phase-plane behavior depends on the dimension; see \cite[Section~2]{JaegerKaul1983}. We record the properties needed for the spectral argument.
\begin{lemma}\label{lem:profile}
    For each $3\le n\le6$, there is a unique regular profile $f$ satisfying
    \begin{equation}\label{eq:profile}
        \begin{cases}
            \ \displaystyle f''+\frac{n-1}{r}f'-\frac{n-1}{r^2}\sin f\cos f=0, \quad 0<r<1,\\
            \ f(0)=0,\qquad f'>0\ \text{on }(0,1),\qquad f'(1)=0.
        \end{cases}
    \end{equation}
    It has a smooth odd extension through the origin and satisfies $f(r)=cr+O(r^3)$ for some $c>0$. Moreover, $0<f(r)<\pi$ for $0<r\le1$ and $f(1)\in(\pi/2,\pi)$. Consequently, $\cos f$ has exactly one simple zero in $(0,1)$, while $\sin f>0$ on $(0,1]$.
\end{lemma}

We next check the weighted Neumann equation satisfied by the coordinates of $U$. The expansion at the origin shows that $\sin f(r)/r$ and $\cos f(r)$ extend smoothly as radial functions. Hence $U$ and $Q$ are smooth on $\overline{\mathbb{B}^n}$, and $Q$ is strictly positive because $Q(0)=nc^2>0$ and $\sin f(r)>0$ for $r>0$. The profile equation and $f'(1)=0$ give
\begin{equation}\label{eq:map-equation}
    -\Delta U=QU\quad\text{in }\mathbb{B}^n,\qquad \partial_\nu U=0\quad\text{on }\partial\mathbb{B}^n.
\end{equation}
For $j=1,\ldots,n+1$, integrating the coordinate equations in \eqref{eq:map-equation} and testing them with $U_j$ give
\begin{equation*}
    \int_{\mathbb{B}^n} QU_j\dd x=0, \qquad \int_{\mathbb{B}^n}|\nabla U_j|^2\dd x=\int_{\mathbb{B}^n} QU_j^2\dd x.
\end{equation*}
Thus $1$ is a positive $Q$-weighted Neumann eigenvalue; we show in \cref{sec:index} that it is the first.
\begin{remark}
    The dimension restriction is visible in logarithmic radius $t=\log r$. Let $h(t)$ be the untruncated regular profile. Then
    \begin{equation*}
        h''+(n-2)h'-(n-1)\sin h\cos h=0.
    \end{equation*}
    The linearization at $h=\pi/2$ has discriminant
    \begin{equation*}
        \Delta=(n-2)^2-4(n-1), \qquad \Delta<0\ (3\le n\le6) \text{ and } \Delta>0\ (n\ge7).
    \end{equation*}
    For $3\le n\le6$, the profile crosses $\pi/2$ before its first stationary radius, giving the smooth Neumann model. For $n\ge7$, it increases to $\pi/2$ without a finite stationary radius \cite[Section~2]{JaegerKaul1983}. The singular harmonic equator map $x\mapsto x/|x|$ gives the sharp high-dimensional comparison \cite[Section~1.1 and Theorem~1.5]{Vinokurov2026}.
\end{remark}

\subsection{Constant continuation and the dilation estimate}
To use the model at arbitrary scales, extend its profile constantly beyond $r=1$:
\begin{equation*}
    \widehat f(r)=
    \begin{cases}
        \ f(r),&0\le r\le1,\\
        \ \vartheta_n,&r\ge1,
    \end{cases}
    \qquad
    \Phi(x)=\left(\sin\widehat f(|x|)\frac{x}{|x|},\cos\widehat f(|x|)\right),
\end{equation*}
with $\Phi(0)=(0,\ldots,0,1)$. Since $f'(1)=0$, the map is $C^1$ on $\R^n$ and smooth on either side of the unit sphere. Its radial density $\widehat Q(r)=|\nabla\Phi(r\theta)|^2$ equals $Q$ for $r\le1$, and equals $(n-1)\sin^2\vartheta_n/r^2$ for $r\ge1$.

We now establish the monotonicity of this density and the energy estimate for dilations of $\Phi$.
\begin{lemma}\label{lem:energy}
    The density $\widehat Q$ is continuous, positive, and strictly decreasing on $(0,\infty)$. Moreover, $\int_{\mathbb{B}^n}Q\dd x=\frac{n(n-1)}{n-2}\omega_n\sin^2\vartheta_n$, and
    \begin{equation}\label{eq:dilation}
        \int_{\mathbb{B}^n}|\nabla(\Phi(x/s))|^2\dd x
        \le\frac{n(n-1)}{n-2}\omega_n\sin^2\vartheta_n
        \qquad(s>0),
    \end{equation}
    with equality exactly when $0<s\le1$.
\end{lemma}
\begin{proof}
    Differentiating $Q=f'(r)^2+\frac{n-1}{r^2}\sin^2 f(r)$ and using \eqref{eq:profile}, we obtain
    \begin{equation*}
        Q'(r)=-\frac{2(n-1)}{r^3}
        \left[(rf'-\sin f\cos f)^2+\sin^4 f\right]<0
        \qquad(0<r<1).
    \end{equation*}
    For $r>1$, $\widehat Q(r)=(n-1)\sin^2\vartheta_n/r^2$ is positive and strictly decreasing, and it agrees with $Q$ at $r=1$ since $f'(1)=0$.

    Set $h(t)=f(e^t)$ and $V(t)=h'(t)^2-(n-1)\sin^2h(t)$ for $t\le0$. The logarithmic profile equation gives
    \begin{equation*}
        V'(t)=-2(n-2)h'(t)^2<0\qquad(t<0).
    \end{equation*}
    Consequently,
    \begin{align*}
        \frac{d}{dr}\bigl[-r^{n-2}V(\log r)\bigr]
        &=r^{n-3}\bigl(-(n-2)V(\log r)-V'(\log r)\bigr)\\
        &=(n-2)r^{n-1}Q(r).
    \end{align*}
    Since $V(\log r)=O(r^2)$ as $r\downarrow0$, integration from $0$ to $R$ yields the radial Pohozaev identity
    \begin{equation}\label{eq:pohozaev}
        \int_{B_R}Q\dd x=-\frac{n\omega_n}{n-2}R^{n-2}V(\log R),
        \qquad 0<R\le1.
    \end{equation}
    At $R=1$, we have $V(0)=-(n-1)\sin^2\vartheta_n$, and hence
    \begin{equation*}
        \int_{\mathbb{B}^n}Q\dd x=\frac{n(n-1)}{n-2}\omega_n\sin^2\vartheta_n.
    \end{equation*}

    Let $E(R)=\int_{B_R}|\nabla\Phi|^2\dd x$. Since $V(\log R)>V(0)$ for $0<R<1$, \eqref{eq:pohozaev} gives $E(R)<R^{n-2}E(1)$. For $R\ge1$, integration of the exterior density gives
    \begin{equation*}
        E(R)=E(1)+n(n-1)\omega_n\sin^2\vartheta_n\int_1^Rr^{n-3}\dd r
        =R^{n-2}E(1).
    \end{equation*}
    Finally, changing variables $y=x/s$ yields
    \begin{equation*}
        \int_{\mathbb{B}^n}|\nabla(\Phi(x/s))|^2\dd x
        =s^{n-2}E(s^{-1})\le E(1)
        =\frac{n(n-1)}{n-2}\omega_n\sin^2\vartheta_n.
    \end{equation*}
    The inequality is strict when $s>1$ and is an equality when $0<s\le1$.
\end{proof}

\section{Ingredients for the sharp comparison}\label{sec:comparison}
We establish the auxiliary results needed for the domain comparison: a center of mass for the coordinates of $\Phi$, a strict mass-transplantation inequality, and the first positive weighted Neumann eigenvalue of the model ball.

\subsection{Center of mass}
We begin with a uniform small-ball estimate for nonatomic measures.
\begin{lemma}\label{lem:nonatomic}
    If $\mu$ is finite, nonatomic, and compactly supported in $\R^n$, then
    \begin{equation*}
        \sup_{c\in\R^n}\mu(B_\varepsilon(c))\longrightarrow0
        \qquad(\varepsilon\downarrow0).
    \end{equation*}
\end{lemma}
\begin{proof}
    Otherwise, there are $\delta>0$, $\varepsilon_j\downarrow0$, and centers $c_j$ such that $\mu(B_{\varepsilon_j}(c_j))\ge\delta$. Because these balls meet the compact support of $\mu$, their centers are bounded. Passing to a subsequence, let $c_j\to c$. For every $\rho>0$, we then have $B_{\varepsilon_j}(c_j)\subset B_\rho(c)$ for all sufficiently large $j$. Thus $\mu(B_\rho(c))\ge\delta$, and continuity from above gives $\mu(\{c\})\ge\delta$, contradicting nonatomicity.
\end{proof}
The uniform estimate makes the averaged height coordinate negative at a sufficiently small scale, independently of the translation center in a fixed ball. It allows one translation and scale to give the full map a center of mass.
\begin{lemma}\label{lem:centering}
    For every nonzero finite compactly supported nonatomic measure $\mu$ on $\R^n$, there exist $c\in\R^n$ and $s>0$ such that
    \begin{equation}\label{eq:center}
        \int_{\R^n}\Phi((x-c)/s)\,d\mu(x)=0.
    \end{equation}
\end{lemma}
\begin{proof}
    Dividing $\mu$ by its total mass does not change \eqref{eq:center}, so assume $\mu(\R^n)=1$. Choose $R_0>0$ with $\supp\mu\subset B_{R_0}$, fix $C>R_0$, and set
    \begin{equation*}
        F(c,s)=\int_{\R^n}\Phi((x-c)/s)\,d\mu(x),\qquad (c,s)\in\R^n\times(0,\infty).
    \end{equation*}
    The continuity and boundedness of $\Phi$ imply, by dominated convergence, that $F$ is continuous.

    Since $\Phi_{n+1}(0)=1$, choose $\delta>0$ such that $\Phi_{n+1}(y)\ge1/2$ for $|y|\le\delta$. If $S>(R_0+C)/\delta$, then $|(x-c)/S|<\delta$ for $x\in\supp\mu$ and $|c|\le C$. Hence
    \begin{equation*}
        F_{n+1}(c,S)\ge\frac12\qquad(|c|\le C).
    \end{equation*}

    For any $s>0$, we have $\Phi_{n+1}((x-c)/s)=\cos\vartheta_n$ when $x\notin B_s(c)$, while $\Phi_{n+1}((x-c)/s)\le1$ inside $B_s(c)$. Thus
    \begin{align*}
        F_{n+1}(c,s)&=\int_{B_s(c)}\cos\widehat f(|x-c|/s)\,d\mu(x)+\int_{\R^n\setminus B_s(c)}\cos\vartheta_n\,d\mu(x)\\
                    &\le\mu(B_s(c))+\cos\vartheta_n\bigl(1-\mu(B_s(c))\bigr).
    \end{align*}
    Since $\cos\vartheta_n<0$ and the small-ball masses tend to zero uniformly in $c$ by \cref{lem:nonatomic}, we may choose $0<\varepsilon<S$ such that
    \begin{equation*}
        F_{n+1}(c,\varepsilon)<0\qquad(|c|\le C).
    \end{equation*}

    Let $F_{\rm hor}$ denote the first $n$ components of $F$. If $|c|=C$ and $x\in\supp\mu$, then $\langle x-c,c\rangle<0$. Since $\sin\widehat f(r)>0$ for $r>0$,
    \begin{equation*}
        \langle F_{\rm hor}(c,s),c\rangle
        =\int_{\R^n}\frac{\sin\widehat f(|x-c|/s)}{|x-c|}
        \langle x-c,c\rangle\,d\mu(x)<0
        \qquad(\varepsilon\le s\le S).
    \end{equation*}

    Set
    \begin{equation*}
        K=\overline{B_C(0)}\times[\varepsilon,S],\qquad
        \mathcal F(c,s)=\bigl(F_{\rm hor}(c,s),-F_{n+1}(c,s)\bigr).
    \end{equation*}
    The preceding inequalities show that $\mathcal F$ points strictly inward on every boundary face of $K$. Let $P_K$ be the Euclidean nearest-point projection onto $K$. The map
    \begin{equation*}
        T(z)=P_K\bigl(z+\mathcal F(z)\bigr),\qquad z\in K,
    \end{equation*}
    is continuous and maps $K$ into itself. Brouwer's fixed-point theorem therefore gives $z\in K$ with $z=P_K(z+\mathcal F(z))$. For any $w\in K$ and $0<\tau\le1$, convexity gives $z+\tau(w-z)\in K$. Since $z$ is the nearest point in $K$ to $z+\mathcal F(z)$,
    \begin{align*}
        0 \le \bigl|\mathcal F(z)-\tau(w-z)\bigr|^2-|\mathcal F(z)|^2 = -2\tau\langle\mathcal F(z),w-z\rangle+\tau^2|w-z|^2.
    \end{align*}
    Dividing by $2\tau$ and letting $\tau\downarrow0$ yields
    \begin{equation*}
        \langle\mathcal F(z),w-z\rangle\le0 \qquad \text{for every }w\in K.
    \end{equation*}
    If $\mathcal F(z)\ne0$, the inward signs on the boundary faces containing $z$ imply $z+\eta\mathcal F(z)\in K$ for all sufficiently small $\eta>0$; this also holds when $z$ lies in the interior. Taking $w=z+\eta\mathcal F(z)$ gives $\eta|\mathcal F(z)|^2\le0$, a contradiction. Hence $\mathcal F(z)=0$, so $F(z)=0$ and \eqref{eq:center} follows.
\end{proof}

\subsection{Mass transplantation}
For a decreasing radial density, the ball maximizes the integral among sets of the same volume. Strict decrease also characterizes equality.
\begin{lemma}\label{lem:rearrange}
    Let $\Omega\subset\R^n$ be measurable with $|\Omega|=|B_R(c)|$. If $g\ge0$ is decreasing and $\int_{B_R(c)}g(|x-c|)\dd x<\infty$, then
    \begin{equation*}
        \int_\Omega g(|x-c|)\dd x\le\int_{B_R(c)}g(|x-c|)\dd x.
    \end{equation*}
    If $g$ is strictly decreasing, equality holds if and only if $|\Omega\mathbin\triangle B_R(c)|=0$.
\end{lemma}
\begin{proof}
    Put $E_-=B_R(c)\setminus\Omega$ and $E_+=\Omega\setminus B_R(c)$. Since $|E_-|=|E_+|$, subtracting the common integral over $\Omega\cap B_R(c)$ gives
    \begin{align*}
        \int_{B_R(c)}&g(|x-c|)\dd x - \int_\Omega g(|x-c|)\dd x \\
                     &= \int_{E_-}\bigl(g(|x-c|)-g(R)\bigr)\dd x + \int_{E_+}\bigl(g(R)-g(|x-c|)\bigr)\dd x \ge 0.
    \end{align*}
    The first term is integrable by assumption, and the second integrand is bounded by $g(R)$. If $g$ is strictly decreasing, both integrands are positive almost everywhere on their respective sets. Thus equality forces $|E_-|=|E_+|=0$, and the converse is immediate.
\end{proof}

\subsection{The first eigenvalue of the model ball}\label{sec:index}
The coordinate equations in \eqref{eq:map-equation} show that $1$ is a $Q$-weighted Neumann eigenvalue. To prove that it is the first positive one, consider the quadratic form
\begin{equation*}
    \qform:H^1(\mathbb{B}^n)\to\R, \qquad \qform[v]=\int_{\mathbb{B}^n}\bigl(|\nabla v|^2-Qv^2\bigr)\dd x.
\end{equation*}
The Morse index of $\qform$ is the maximal dimension of a subspace on which it is negative definite; its full $H^1(\mathbb{B}^n)$ domain gives the natural Neumann boundary condition.
\begin{proposition}\label{prop:index}
    The form $\qform$ has Morse index one, and its associated bilinear form has kernel
    \begin{equation*}
        \Span\{\cos f,\sin f\,\theta_1,\ldots,\sin f\,\theta_n\}.
    \end{equation*}
\end{proposition}
\begin{proof}
    Let $v\in H^1(\mathbb{B}^n)$ and write $x=r\theta$, where $0<r<1$ and $\theta\in\Sph^{n-1}$. Let $\{Y_{\ell,k}\}_{\ell\ge0,\,1\le k\le d_\ell}$ be a real orthonormal basis of spherical harmonics, where $d_\ell$ is the dimension of the degree-$\ell$ space. Then
    \begin{equation*}
        -\Delta_{\Sph^{n-1}}Y_{\ell,k}=\ell(\ell+n-2)Y_{\ell,k},\qquad
        \int_{\Sph^{n-1}}Y_{\ell,k}Y_{m,j}\dd\sigma=\delta_{\ell m}\delta_{kj}.
    \end{equation*}
    For almost every $r$, the expansion in $L^2(\Sph^{n-1})$ is
    \begin{equation*}
        v(r\theta)=\sum_{\ell,k}a_{\ell,k}(r)Y_{\ell,k}(\theta),\qquad
        a_{\ell,k}(r)=\int_{\Sph^{n-1}}v(r\eta)Y_{\ell,k}(\eta)\dd\sigma(\eta).
    \end{equation*}
    Since $Q=Q(r)$, orthogonality gives
    \begin{equation}\label{eq:radial-form}
        \begin{aligned}
            \qform[v]=\sum_{\ell,k}\qform_\ell[a_{\ell,k}], \qquad \qform_\ell[a]=\int_0^1\left(a'^2+\frac{\ell(\ell+n-2)}{r^2}a^2-Q(r)a^2\right)r^{n-1}\dd r.
        \end{aligned}
    \end{equation}
    By \eqref{eq:map-equation}, the radial equations in degrees zero and one are
    \begin{align*}
        -(r^{n-1}(\cos f(r))')'&=r^{n-1}Q(r)\cos f(r),\\
        -(r^{n-1}(\sin f(r))')'+(n-1)r^{n-3}\sin f(r)&=r^{n-1}Q(r)\sin f(r).
    \end{align*}
    Testing these equations with $\cos f(r)w(r)^2$ and $\sin f(r)w(r)^2$, respectively, and integrating by parts give, for smooth $w$,
    \begin{equation}\label{eq:radial-ground-state}
        \begin{aligned}
            \qform_0[\cos f(r)w(r)]&=\int_0^1\cos^2 f(r)|w'(r)|^2r^{n-1}\dd r,\\
            \qform_1[\sin f(r)w(r)]&=\int_0^1\sin^2 f(r)|w'(r)|^2r^{n-1}\dd r.
        \end{aligned}
    \end{equation}
    The boundary terms vanish by $f'(1)=0$ and $f(r)=f'(0)r+O(r^3)$ as $r\downarrow0$.

    By \cref{lem:profile}, $\cos f(r)$ has a unique simple zero $r_0\in(0,1)$. If a smooth radial function $a$ satisfies $a(r_0)=0$, then $w(r)=a(r)/\cos f(r)$ extends smoothly across $r_0$. The first identity in \eqref{eq:radial-ground-state} gives
    \begin{equation*}
        \qform_0[a]=\int_0^1\cos^2 f(r)\left|\left(\frac{a(r)}{\cos f(r)}\right)'\right|^2r^{n-1}\dd r\ge0.
    \end{equation*}
    Point evaluation at $r_0>0$ is continuous in the radial $H^1$ norm, so the inequality extends by approximation to the codimension-one subspace $a(r_0)=0$. On the other hand,
    \begin{equation*}
        \qform_0[1]=-\int_0^1Q(r)r^{n-1}\dd r<0.
    \end{equation*}
    Constants give a negative direction, while every two-dimensional subspace contains a nonzero function with $a(r_0)=0$. Thus $\qform_0$ has Morse index one.

    In degree one, a smooth radial coefficient satisfies $a(r)=a'(0)r+O(r^3)$, while $\sin f(r)=f'(0)r+O(r^3)>0$ for $0<r\le1$. Thus $w(r)=a(r)/\sin f(r)$ extends smoothly across the origin. The second identity in \eqref{eq:radial-ground-state} gives
    \begin{equation*}
        \qform_1[a]=\int_0^1\sin^2 f(r)\left|\left(\frac{a(r)}{\sin f(r)}\right)'\right|^2r^{n-1}\dd r\ge0.
    \end{equation*}
    Approximation gives $\qform_1\ge0$ on its full form domain. In either degree, a radial kernel function solves the corresponding second-order equation with $a'(1)=0$. The equation is regular at $r=1$, so the solution is determined by $a(1)$. As $\cos f(1)\sin f(1)\ne0$, the radial kernels are $\Span\{\cos f\}$ and $\Span\{\sin f\}$, respectively.

    All positive degrees have the same radial form domain, since the angular term in \eqref{eq:radial-form} requires $\int_0^1a(r)^2r^{n-3}\dd r<\infty$. For $\ell\ge2$, the nonnegativity of $\qform_1$ yields
    \begin{equation*}
        \qform_\ell[a] = \qform_1[a]+\bigl(\ell(\ell+n-2)-(n-1)\bigr)\int_0^1a(r)^2r^{n-3}\dd r>0 \qquad (a\ne0).
    \end{equation*}
    By completeness of the spherical-harmonic expansion, degree zero contributes the only negative direction and the zero mode $\cos f$, while degree one contributes the $n$ zero modes $\sin f(r)\theta_1,\ldots,\sin f(r)\theta_n$.
\end{proof}
We now use the index count to locate the coordinate eigenvalue in the weighted Neumann spectrum.
\begin{corollary}\label{cor:ball-spectrum}
    The density $Q$ is smooth and strictly positive on $\overline{\mathbb{B}^n}$, and
    \begin{equation}\label{eq:ball-value}
        \lambda_1^N(\mathbb{B}^n,Q\dd x)=1,\qquad \int_{\mathbb{B}^n}Q\dd x=\frac{n(n-1)}{n-2}\omega_n\sin^2\vartheta_n.
    \end{equation}
    Its first positive eigenspace is the $(n+1)$-dimensional kernel in \cref{prop:index}.
\end{corollary}
\begin{proof}
    The regular profile makes $Q$ smooth and positive on $\overline{\mathbb{B}^n}$, so the weighted Neumann spectrum is discrete; \cref{lem:energy} gives $\int_{\mathbb{B}^n}Q\dd x=\frac{n(n-1)}{n-2}\omega_n\sin^2\vartheta_n$. By \eqref{eq:map-equation}, every coordinate $U_j$ solves the Neumann problem with eigenvalue $1$. Integrating its equation gives $\int_{\mathbb{B}^n}QU_j\dd x=0$. Since $U_j$ is nonconstant, $1$ is a positive eigenvalue.

    Suppose that there is an eigenvalue $0<\lambda<1$ with eigenfunction $\varphi$. Testing its equation with $1$ and $\varphi$ gives
    \begin{equation*}
        \int_{\mathbb{B}^n}Q\varphi\dd x=0,\qquad
        \int_{\mathbb{B}^n}|\nabla\varphi|^2\dd x
        =\lambda\int_{\mathbb{B}^n}Q\varphi^2\dd x.
    \end{equation*}
    The first equality eliminates the mixed term in $\qform[a+b\varphi]$. Thus, for $(a,b)\ne(0,0)$,
    \begin{equation*}
        \qform[a+b\varphi]
        =-a^2\int_{\mathbb{B}^n}Q\dd x
        +b^2(\lambda-1)\int_{\mathbb{B}^n}Q\varphi^2\dd x<0.
    \end{equation*}
    The functions $1$ and $\varphi$ are linearly independent, so their span would be a two-dimensional negative definite subspace, contradicting \cref{prop:index}. Hence no eigenvalue lies in $(0,1)$. Since $\mathbb{B}^n$ is connected, the eigenvalue $0$ consists only of constants; the coordinate eigenfunctions therefore show that $\lambda_1^N(\mathbb{B}^n,Q\dd x)=1$.

    Finally, the weak equation for an eigenfunction at eigenvalue $1$ is
    \begin{equation*}
        \int_{\mathbb{B}^n}\langle\nabla\varphi,\nabla\psi\rangle\dd x
        =\int_{\mathbb{B}^n}Q\varphi\psi\dd x
        \qquad(\psi\in H^1(\mathbb{B}^n)).
    \end{equation*}
    Thus the bilinear form associated with $\qform$ vanishes on $(\varphi,\psi)$ for every $\psi\in H^1(\mathbb{B}^n)$. By \cref{prop:index}, its kernel is the stated $(n+1)$-dimensional eigenspace.
\end{proof}

\section{Proof of main results}\label{sec:main-proof}
We now combine the preceding results to prove the sharp weighted Neumann inequality. We then characterize equality and establish the Steklov corollary.

\subsection{The sharp weighted Neumann bound}\label{subsec:main-inequality}
By spatial dilation, assume $|\Om|=\omega_n$. The case $\mu=0$ is immediate. Otherwise, extend $\mu$ by zero to $\R^n$, choose $(c,s)$ by \cref{lem:centering}, and set $W(x)=\Phi((x-c)/s)$. Every coordinate of $W$ belongs to $\mathscr D(\Om)$ and has zero $\mu$-mean. Since $|W|=1$ and $\mu(\overline\Om)>0$, at least one coordinate has positive weighted norm and a finite Rayleigh quotient. Thus $\lambda_1^N(\Om,\mu)<\infty$, and \eqref{eq:rayleigh} gives
\begin{equation*}
    \lambda_1^N(\Om,\mu)\int_{\overline\Om}W_j^2\,d\mu \le\int_\Om|\nabla W_j|^2\dd x.
\end{equation*}
For a coordinate with zero weighted norm, this inequality follows from the nonnegativity of its energy. Summing over the coordinates and using $\sum_{j=1}^{n+1}W_j^2=1$, we obtain the first inequality in
\begin{equation}\label{eq:comparison-chain}
    \overline\lambda_1^N(\Om,\mu) \le\int_\Om|\nabla W|^2\dd x \le\int_{B_1(c)}|\nabla W|^2\dd x \le\frac{n(n-1)}{n-2}\omega_n\sin^2\vartheta_n.
\end{equation}
The density $|\nabla W|^2=s^{-2}\widehat Q(|x-c|/s)$ is strictly decreasing in $|x-c|$, so \cref{lem:rearrange} gives the second inequality. The last follows from \eqref{eq:dilation} after translation. Restoring the volume proves \eqref{eq:main-bound}.

By \cref{cor:ball-spectrum}, the ball with density $Q$ attains the bound in \eqref{eq:comparison-chain}. Taking the supremum over measures and then over admissible domains proves the optimization identity in \cref{thm:main}.

\subsection{Rigidity in \cref{thm:main}}\label{subsec:main-equality}
Suppose equality holds in \eqref{eq:main-bound}. By spatial dilation, assume $|\Om|=\omega_n$. Then $\mu\ne0$ and $\lambda_1^N(\Om,\mu)>0$. Rescale $\mu$ so that $\lambda_1^N(\Om,\mu)=1$; consequently, $\mu(\overline\Om)=\frac{n(n-1)}{n-2}\omega_n\sin^2\vartheta_n$. Choose $(c,s)$ from \cref{lem:centering} and set $W(x)=\Phi((x-c)/s)$. Equality in \eqref{eq:main-bound} forces equality throughout \eqref{eq:comparison-chain}. By \cref{lem:energy} and \cref{lem:rearrange}, this gives
\begin{equation*}
    |\Om\mathbin\triangle B_1(c)|=0,\qquad 0<s\le1.
\end{equation*}
Each coordinate satisfies the Rayleigh inequality used above. Since their nonnegative differences sum to zero, equality holds for every coordinate:
\begin{equation*}
    \int_\Om|\nabla W_j|^2\dd x=\int_{\overline\Om} W_j^2\,d\mu,\qquad j=1,\ldots,n+1.
\end{equation*}
For any $\eta\in\mathscr D(\Om)$ with zero $\mu$-mean, the Rayleigh inequality applied to $W_j+t\eta$ is an equality at $t=0$. Its first variation therefore vanishes. Subtracting the $\mu$-mean of a general $\eta$ does not change the resulting identity, because $\int_{\overline{\Omega}} W_j\,d\mu=0$. Hence
\begin{equation}\label{eq:weak-equality}
    \int_\Om\nabla W_j\cdot\nabla\eta\dd x =\int_{\overline\Om} W_j\eta\,d\mu \qquad(\eta\in\mathscr D(\Om)).
\end{equation}

Let $\zeta$ be smooth on a neighborhood of $\overline\Om$. Taking $\eta=\zeta W_j$ in \eqref{eq:weak-equality} and summing over $j$, we use $\sum_jW_j^2=1$ and $\sum_jW_j\nabla W_j=0$ to obtain
\begin{equation*}
    \int_{\overline\Om}\zeta\,d\mu=\int_\Om\zeta|\nabla W|^2\dd x.
\end{equation*}
Such smooth functions are uniformly dense in $C^0(\overline\Om)$, so the identity determines both finite Radon measures on $\overline\Om$. Substituting the resulting measure into \eqref{eq:weak-equality} gives
\begin{equation}\label{eq:measure-map}
    \mu=\mathbf1_\Om|\nabla W|^2\dd x,\qquad -\Delta W=|\nabla W|^2W\quad\text{weakly in }\Om.
\end{equation}
In particular, $\mu$ has no boundary-supported part.

We claim that $s=1$. Otherwise, $\Om\cap\{s<|x-c|<1\}$ is a nonempty open set because $\Om$ agrees almost everywhere with $B_1(c)$. Choose a nonzero nonnegative $\eta\in C_c^\infty(\Om\cap\{s<|x-c|<1\})$. On its support, $W_{n+1}=\cos\vartheta_n<0$, whereas $|\nabla W|^2=(n-1)\sin^2\vartheta_n/|x-c|^2>0$. Taking $j=n+1$ in \eqref{eq:weak-equality} and using \eqref{eq:measure-map}, we obtain
\begin{equation*}
    0=\int_{\overline\Om} W_{n+1}\eta\,d\mu=\cos\vartheta_n\int_\Om\eta|\nabla W|^2\dd x<0,
\end{equation*}
a contradiction. Thus $s=1$.

Since $\Om$ is open, the almost-everywhere identity implies $\Om\subset B_1(c)$: otherwise, a neighborhood of a point in $\Om\setminus B_1(c)$ would contain a positive-measure subset of $\Om\setminus B_1(c)$. The same identity makes $\Om$ dense in $B_1(c)$, so $\overline\Om=\overline{B_1(c)}$. Since $s=1$, we have $W=U(x-c)$ on this closure. By \eqref{eq:measure-map}, $\mu=Q(|x-c|)\dd x$ under the chosen normalization; restoring the measure scale gives an arbitrary positive multiple.

If $\Om$ has Lipschitz boundary and $x\in B_1(c)\setminus\Om$, then $x\in\partial\Om$. The exterior density property of $\Om$ gives, for all sufficiently small $\varepsilon>0$ with $B_\varepsilon(x)\subset B_1(c)$,
\begin{equation*}
    0<|B_\varepsilon(x)\setminus\Om|\le|B_1(c)\setminus\Om|=0,
\end{equation*}
a contradiction. Therefore $\Om=B_1(c)$.

For general volume $|\Om|=\omega_nR^n$, spatial dilation gives, for some $c\in\R^n$ and $a>0$,
\begin{equation*}
    |\Om\mathbin\triangle B_R(c)|=0,\qquad \overline\Om=\overline{B_R(c)},\qquad \mu=aQ(|x-c|/R)\dd x.
\end{equation*}
If $\Om$ is Lipschitz, then $\Om=B_R(c)$. Conversely, for $\Om=B_R(c)$ and $\mu=aQ(|x-c|/R)\dd x$, changing variables $x=c+Ry$ and using \cref{cor:ball-spectrum} give
\begin{equation*}
    \lambda_1^N(\Om,\mu)=\frac{1}{aR^2},\qquad
    \mu(\overline\Om)=aR^n\frac{n(n-1)}{n-2}\omega_n\sin^2\vartheta_n.
\end{equation*}
Thus $\overline\lambda_1^N(\Om,\mu)=\frac{n(n-1)}{n-2}\omega_n\sin^2\vartheta_n R^{n-2}$, as required. Without the Lipschitz condition, the necessity statement determines $\Om$ almost everywhere and determines its closure; null modifications of an open set need not preserve its Sobolev form domain.

\subsection{Proof of \cref{cor:steklov}}\label{sec:steklov}
Let $\Om$ be admissible with $|\partial\Om|<\infty$ and take $\mu=\mathcal H^{n-1}\!\restriction_{\partial\Om}$. This measure is finite and nonatomic. Applying \eqref{eq:main-bound} and \eqref{eq:steklov-measure} gives the non-strict form of \eqref{eq:steklov-bound}. Equality would make $\mu$ a positive interior density by the preceding equality analysis, although it is concentrated on $\partial\Om$. Therefore \eqref{eq:steklov-bound} is strict. This also covers disconnected domains, whose first variational level after the constants may be zero.

For sharpness, \eqref{eq:ball-value} and the upper bound show that the supremum over nonnegative $L^1$ densities on $\mathbb{B}^n$ is $\frac{n(n-1)}{n-2}\omega_n\sin^2\vartheta_n$, attained by $Q$. The Steklov approximation theorem of Girouard, Karpukhin, and Lagac\'e \cite{GirouardKarpukhinLagace2021}, in the form of \cite[Proposition~1.4]{Vinokurov2026}, supplies bounded $C^1$ domains $\Om_j\subset \mathbb{B}^n$ with
\begin{equation*}
    |\Om_j|\to\omega_n,\qquad \sigma_1(\Om_j)|\partial\Om_j|\to\frac{n(n-1)}{n-2}\omega_n\sin^2\vartheta_n.
\end{equation*}
Multiplying by $|\Om_j|^{(2-n)/n}$ gives the sharp limit in \cref{cor:steklov}. The constant is approached by $C^1$ perforated domains but is not attained by an admissible domain with finite boundary measure.

\section*{Acknowledgments}
The authors acknowledge the use of AI to assist with English-language editing, manuscript polishing, and preliminary exploration of mathematical proofs. The authors take full responsibility for all mathematical content, arguments, and the final version of the manuscript.

\bibliographystyle{amsalpha-author-sorted}
\bibliography{CLY-Q1.12-260926}

\vspace{1cm}

\end{document}